\documentclass{amsart}
\usepackage[psamsfonts]{amssymb}
\usepackage{amsmath,amsthm}

\usepackage{enumerate}
\usepackage{multirow}

\newtheorem{thm}{Theorem}[section]
\newtheorem{cor}[thm]{Corollary}
\newtheorem{lem}[thm]{Lemma}
\newtheorem{prop}[thm]{Proposition}
\newtheorem{exam}[thm]{Example}

\theoremstyle{remark}
\newtheorem{rem}{Remark}[section]

\newcommand{\tr}{{\mathsf{T}}}

\def\f{\frac}

 \def\a{{\alpha}}
 \def\b{{\beta}}

 \def\l{{\lambda}}
 \def\d{{\delta}}

 \def\la{{\langle}}
 \def\ra{{\rangle}}

 \def\eb{\mathbf{e}}

 \def\Ab{\mathbf{A}}
 \def\Bb{\mathbf{B}}
 \def\Cb{\mathbf{C}}
 \def\Db{\mathbf{D}}

 \def\Hb{\mathbf{H}}
 \def\Ib{\mathbf{I}}

 \def\CV{\mathcal{V}}

 \def\PP{\mathbb{P}}
 \def\QQ{\mathbb{Q}}
 \def\RR{\mathbb{R}}
 
 \def\SS{\mathbb{S}}

\def\c={\stackrel{{\scriptstyle{\mathrm{c}}}}{=}}

\newcommand{\wt}{\widetilde}
\newcommand{\wh}{\widehat}

\begin{document}

\title[Sobolev orthogonal polynomials on product domains]
{Sobolev orthogonal polynomials on product domains}

\author[L. Fern\'{a}ndez]{Lidia Fern\'{a}ndez}
\address[L. Fern\'{a}ndez]{Departamento de Matem\'atica Aplicada,
Universidad de Granada, Spain}

\author[F. Marcell\'{a}n]{Francisco Marcell\'{a}n}
\address[F. Marcell\'{a}n]{Instituto de Ciencias Matem\'aticas (ICMAT) and Departamento de Matem\'aticas,
Universidad Carlos III de Madrid, Spain}

\author[T. E. P\'erez]{Teresa E. P\'erez}
\address[T. E. P\'erez]{Departamento de Matem\'atica Aplicada,
Universidad de Granada, Spain}

\author[M. A. Pi\~{n}ar]{Miguel A. Pi\~{n}ar}
\address[M. A. Pi\~{n}ar]{Departamento de Matem\'atica Aplicada,
Universidad de Granada, Spain}

\author[Y. Xu]{Yuan Xu}
\address[Y. Xu]{Department of Mathematics,
Universty of Oregon, USA}

\thanks{The work of the first, third and fourth author has been partially supported by DGICYT, Ministerio de Econom\'ia y Competitividad (MINECO) of Spain grant MTM 2011--28952--C02--02. The work of the second author has been supported by DGICYT, Ministerio de Econom\'ia y Competitividad (MINECO) of Spain grant MTM2012-36732-C03-01. The work of the fifth author was supported in part by NSF Grant DMS-1106113}

\date{\today}

\keywords{classical orthogonal polynomials, orthogonal polynomials in two variables, Sobolev inner products, product domain}
\subjclass[2000]{33C50, 42C10}

\begin{abstract}
Orthogonal polynomials on the product domain $[a_1,b_1] \times [a_2,b_2]$ with respect to the inner product
$$
 \la f ,g \ra_S =  \int_{a_1}^{b_1} \int_{a_2}^{b_2} \nabla f(x,y)\cdot \nabla g(x,y)\, w_1(x)w_2(y)
 \,dx\, dy + \l f(c_1,c_2)g(c_1,c_2)
$$
are constructed, where  $w_i$ is a weight function on $[a_i,b_i]$ for $i = 1, 2$, $\l > 0$, and $(c_1, c_2)$
is a fixed point. The main result shows how an orthogonal basis for
such an inner product can be constructed for certain weight functions, in particular, for product Laguerre
and product Gegenbauer weight functions, which serve as primary examples.
\end{abstract}

\maketitle

\section{Introduction}
\setcounter{equation}{0}

Let $w_i(x)$ be a nonnegative weight function defined on an interval $[a_i,b_i]$, where $i =1,2$. Let $W$ be
the product weight function
\begin{equation} \label{eq:W}
   W(x,y):= w_1(x) w_2(y), \qquad  (x,y) \in \Omega : = [a_1,b_1] \times [a_2,b_2].
\end{equation}
The purpose of this paper is to study orthogonal polynomials with respect to the inner product
\begin{equation} \label{eq:ipd-1}
   \la f ,g \ra_S =  \iint\limits_{\Omega} \nabla f(x,y)\cdot \nabla g(x,y)\, W(x,y) \,dx\, dy + \l f(c_1,c_2)g(c_1,c_2),
\end{equation}
where $\l > 0$ and $(c_1,c_2)$ is a fixed point, typically a corner point of  the product domain $\Omega$.

Sobolev orthogonal polynomials of one variable have been extensively studied (see the survey \cite{MX}). 
In particular, polynomials that are orthogonal with respect to the one--variable analogue of the inner product
\eqref{eq:W} were analyzed in \cite{KLJ97}. In
contrast, the study of such polynomials in several variables is a fairly recent affair. In \cite{X08}, one of the earliest
studies in several variables, Sobolev orthogonal polynomials with respect to an inner product similar to \eqref{eq:ipd-1} on
the unit ball of $\RR^d$ are constructed, where the discrete part could also be replaced by the integral on the boundary
of the ball. The
motivation of \cite{X08} came from a question from engineering that requires control over the gradient. Such
inner products appear naturally in the analysis of spectral methods for numerical solutions of partial differential equations
(cf. \cite{LX}), which motivates our study.

For the ordinary inner product on the product domain,
\begin{equation} \label{eq:ipd-2}
 \la f ,g \ra_W =  \iint\limits_{\Omega} f(x,y) g(x,y)\, W(x,y) \,dx\, dy,
\end{equation}
it is immediate that a basis of orthogonal polynomials of degree $n$ is given by
$p_k(w_1;x)p_{n-k}(w_2;y)$, $0\leqslant k \leqslant n$, where $p_k(w;x)$ denotes the orthogonal polynomial of degree $k$ with
respect to $w$.  A moment reflection shows, however, that Sobolev orthogonal polynomials with respect to the
inner product \eqref{eq:ipd-1} do not possess product structure. Our goal in this paper is to study  the orthogonal
structure for the inner product \eqref{eq:ipd-1} on the product domain.

Our main result provides a way to construct a basis of Sobolev orthogonal polynomials, complemented with an
algorithm that computes both orthogonal polynomials and their $L^2$ norm, when both weight functions $w_1$
and $w_2$ are self-coherent, which means that their monic orthogonal polynomials satisfy the relations of the form
\begin{equation} \label{eq:cohen-poly}
  p_n (x)= \frac{p'_{n+1}(x)}{n+1} + a_n  p'_{n} (x) + b_n  p'_{n-1}(x), \qquad n \geqslant 1.
\end{equation}
Weight functions, or measures, that are self-coherent have been studied extensively and characterized. They are essentially the classical measures. In \cite{MBP} the authors proved that \eqref{eq:cohen-poly} characterizes classical orthogonal polynomials.

Our approach is to express the Sobolev orthogonal polynomials with respect to the inner product $\la \cdot,\cdot \ra_S$
in terms of a family of product polynomials, which are not, however, the product orthogonal polynomials with respect to \eqref{eq:ipd-2}, but product polynomials of the form $q_k(w_1;x) q_{n-k}(w_2;y)$, where $q_k(w)$ takes the form of
the right hand side of \eqref{eq:cohen-poly} without the derivative. In order to keep the idea transparent, we will not
work with the most general case that our method applies, but work primarily with two examples, product
Laguerre weight functions and product Gegenbauer weight functions, for which we work out our algorithms explicitly.

Some of our results can no doubly be extended from two variables to several variables. We choose to stay
with two variables to avoid complicated notation and keep the algorithm practical.

The paper is organized as follows. In the next section, we recall the basics for orthogonal polynomials of several
variables, and describe our strategy for constructing Sobolev orthogonal polynomials for the product
weight functions. The construction is worked out explicitly in the case of product Laguerre weight in
Section 3 and in the case of product Gegenbauer weight in Section 4.

\section{Constructing bases for Sobolev orthogonal polynomials}
\setcounter{equation}{0}

The basics of orthogonal polynomials in several variables are given in the first subsection. Sobolev
orthogonal polynomials for product measures are described in the second subsection, and the
strategy for constructing an orthogonal basis is discussed in the third subsection.

\subsection{Orthogonal polynomials of two variables}
Let $\Pi^2$ denote the space of polynomials in two real variables and, for $n = 0,1,2,\ldots$, let $\Pi_n^2$ denote the
subspace of polynomials of (total) degree at most $n$ in $\Pi^2$. For an inner product $\la \cdot, \cdot \ra$ defined
on $\Pi^2$, a polynomial $P \in \Pi_n^2$ is said to be orthogonal if $\la P, Q\ra =0$ for all $Q \in \Pi_{n-1}^2$.
Let $\CV_n^2$ denote the space of orthogonal polynomials of total degree $n$ with respect to $\la \cdot, \cdot \ra$.
It is known that
$$
  \dim \Pi_n^2 = \binom{n+2}{n} \quad \hbox{and} \quad \dim \CV_n^2 =  n+1.
$$

The space $\CV_n^2$ can have many different bases. A basis $\{P_k^n: 0\leqslant k\leqslant n\}$ of $\CV_n^2$
is called mutually orthogonal if $\la P_k^n, P_j^n\ra =0$ for $k \ne j$ and it is called orthonormal if, in addition,
$\la P_k^n, P_k^n\ra =1$. Another polynomial basis that is of interest is the monic basis, for which $P_k^n(x,y) = x^{n-k} y^k
+ R_k^n(x,y)$, where $R_k^n \in \Pi_{n-1}^2$, $0\leqslant k \leqslant n$. It is often convenient to use the vector notation
$$
     \PP_n = \big(P^n_{0}, P^n_{1}, \ldots,  P^n_{n} \big)^{\tr},
$$
considered as a column vector, which we also regard as a set of orthogonal polynomials of degree $n$. In this
notation, $\la \PP_n, \PP_m^\tr \ra =
\mathbf{H}_n \delta_{n,m}$, where $\mathbf{H}_n$ is a matrix of size $(n+1) \times (n+1)$, necessarily symmetric and positive definite. If the set $\PP_n$ contains a mutually orthogonal basis then $\mathbf{H}_n$
is diagonal, and if it is orthonormal then $\mathbf{H}_n$ is the identity matrix.

For $W(x,y) = w_1(x) w_2(y)$ as in \eqref{eq:W},  we consider the inner product
$$
   \la f, g\ra_W = c \int_\Omega f(x,y) g(x,y) W(x,y) dxdy,
$$
where $c$ is a normalization constant of $W$ so that $\la 1, 1 \ra_W =1$. A basis of $\CV_n^2 (W)$ is given by
the product polynomials
\begin{equation} \label{eq:productOP}
  P_k^n(x,y) := p_{n-k}(w_1;x)p_k(w_2;y), \qquad 0 \leqslant k \leqslant n,
\end{equation}
where $p_k(w_i;x) = x^k+ \ldots$ denotes the monic orthogonal polynomial with respect to $w_i$ on $[a_i,b_i]$.
Then $P_k^n$ is the monic orthogonal polynomial and $\{P_k^n: 0\leqslant k \leqslant n\}$ forms a mutually orthogonal basis of
$\CV_n^2(W)$.

\subsection{Sobolev orthogonal polynomials}

For $i =1,2$, let $w_i$ be a weight function defined on the interval $[a_i,b_i]$, where $-a_i$ and $b_i$ can be
infinity. For the product weight function $W$ in \eqref{eq:W}, let $\CV_n^2(S)$ denote the space of Sobolev
orthogonal polynomials of degree $n$ with respect to the inner product $\la \cdot,\cdot \ra_S$ defined in
\eqref{eq:ipd-1}. Most of our work will be carried out for the following two examples.

\begin{exam} \label{ex:Laguerre}
For $\a > -1$, let $w_\a$ be the Laguerre weight function
$$
     w_\a(x):= x^\a e^{-x},  \qquad x \in \RR_+:= [0,\infty).
$$
For $\a,\b > -1$, let $W_{\a,\b}$ be the product Laguerre weight function defined by
$$
   W_{\a,\b} (x,y):= w_\a(x)w_\b(y), \qquad (x,y) \in \Omega: = \RR_+^2.
$$
There is only one finite corner point of $\Omega$, and we consider the inner product
\begin{equation} \label{eq:Laguerre}
   \la f, g \ra_{S}=  c_{\a,\b} \int_{\RR_+^2} \nabla f(x,y)\cdot \nabla g(x,y)\, W_{\a,\b}(x,y) \,dx\, dy + \l f(0,0)g(0,0),
\end{equation}
where $\l > 0$ is a fixed constant and $c_{\a,\b}= 1/ \int_{\RR_+^2}  W_{\a,\b}(x,y) \,dx\, dy $.
\end{exam}

\begin{exam} \label{ex:Gegen}
For $\a > -1/2$, let $u_\a$ be the Gegenbauer weight function
$$
     u_\a(x):= (1-x^2)^{\a - \f12},  \qquad x \in [-1,1].
$$
For $\a,\b > -1/2$, let $U_{\a,\b}$ be the product Gegenbauer weight function defined by
$$
   U_{\a,\b} (x,y):= u_\a(x)u_\b(y), \qquad (x,y) \in \Omega: = [-1,1]^2.
$$
There are four corner points of $\Omega$ and we consider the inner product
\begin{align} \label{eq:Gegenbauer}
   \la f, g \ra_{S}= c_{\a,\b} \int_{-1}^1 \int_{-1}^1 \nabla f(x,y)\cdot \nabla g(x,y)\, U_{\a,\b}(x,y) \,dx\, dy + \l f(1,1)g(1,1),
\end{align}
where $\l > 0$ is a fixed constant and  $c_{\a,\b}= 1/ \int_{\Omega}  U_{\a,\b}(x,y) \,dx\, dy $.
\end{exam}

For the inner product $\la \cdot,\cdot \ra_S$ in \eqref{eq:ipd-1}, we denote its main part by
\begin{align}\label{eq:ipd-nabla}
  \la f,g \ra_\nabla : = & c \int_\Omega \nabla f(x,y) \cdot \nabla g(x,y) W(x,y) dxdy \\
      = & \la \partial_1f , \partial_1 g \ra_W +  \la \partial_2 f , \partial_2 g \ra_W.  \notag
\end{align}
This is a bilinear form and it is an inner product on the linear space $\Pi^2 \backslash \RR$ of polynomials
having a zero constant term. Let
$$
    \CV_n^2(S):= \CV_n^2(S,W) \quad \hbox{and} \quad \CV_n^2(\nabla):= \CV_n^2(\nabla,W)
$$
denote the linear spaces of orthogonal polynomials of total degree $n$ associated with $\la \cdot,\cdot \ra_S$
and $\la \cdot,\cdot \ra_\nabla$, respectively.

Let $\mathsf{S}_k^n$ be the monic orthogonal polynomial of degree $n$ in $\CV_n^2(S)$ that satisfies
$\mathsf{S}_k^n(x,y) - x^{n-k}y^k \in \Pi_{n-1}^2$ for $0 \leqslant k \leqslant n$. Likewise, for $n \geqslant 1$,
let $S_k^n$ be a monic orthogonal polynomial in $\CV_n^2(\nabla)$.

\begin{thm}\label{sobolev-basis}
For $n\geqslant 1$, let $\{S_k^n: 0\leqslant k \leqslant n\}$ denote a monic orthogonal basis of
$\CV_n^2(\nabla)$. Then, the monic orthogonal basis $\{\mathsf{S}_k^n: 0 \leqslant k \leqslant n\}$ of $\CV_n^2(S)$
is given by $\mathsf{S}_0^0(x,y) =1$ and
$$
    \mathsf{S}_k^n(x,y) = S_k^n(x,y) -   S_k^n(c_1,c_2), \quad n\geqslant 1.
$$
\end{thm}

\begin{proof}
Since $\mathsf{S}_k^n(c_1,c_2) =0$, it follows that $\la \mathsf{S}_k^n, \mathsf{S}_j^m \ra_S = \la \mathsf{S}_k^n, S_j^m \ra_\nabla$ if $n \geqslant 1$.
\end{proof}

This theorem shows that we only need to work with the bilinear form $\la \cdot, \cdot \ra_\nabla$ and on the linear space $\Pi^2 \backslash \RR$. Observe that the orthogonal polynomials in $\CV_n^2(\nabla)$ are determined up to an additive constant
$c$. Indeed, for any constant $c$, the polynomial $S_k^n + c$ is also a monic orthogonal polynomial in $\CV_n^2(\nabla)$.
By Theorem \ref{sobolev-basis}, however, we only need to determine $S_k^n$ up to a constant. For convenience, we adopt
the following notation for two functions that are equal up to a constant:
$$
   f(x,y) \c= g(x,y) \qquad \hbox{if} \quad   f(x,y) - g(x,y) \equiv c,
$$
where $c \in \RR$ is a generic constant.

\subsection{Strategy for constructing Sobolev orthogonal polynomials}
In order to construct the polynomial $S_k^n$, we expand it in terms of a known basis of polynomials denoted by
$\{Q_j^m: 0\leqslant j \leqslant m \leqslant n\}$,
\begin{equation} \label{eq:Skn=Qjm}
  S_k^n (x,y) = \sum_{m=0}^n \sum_{j=0}^m a_{j,m}(k) Q_j^m (x,y),
\end{equation}
and determine the coefficients $a_{j,m}(k)$ by orthogonality. Since $S_k^n$ is determined up to a constant, the equal sign should be replaced by $\c=$ in \eqref{eq:Skn=Qjm}.

The choice of $Q_j^m$ clearly matters. An
obvious choice is the basis of product orthogonal polynomials $P_k^n$ in \eqref{eq:productOP}. This basis,
however, is not a good choice since we need to work with derivatives of the basis elements. This is where the
notion of coherent pair comes in.

A weight function $w$ defined on the real line is called self-coherent if its monic orthogonal polynomials $p_n(w)$
satisfy the relation
\begin{equation}\label{eq:coherent}
   p_n(w;x) = \frac{p_{n+1}'(w;x)}{n+1} + a_n p_{n}'(w;x), \qquad n \geqslant 0,
\end{equation}
for some constants $a_n$. Furthermore, $w$ is called symmetric self-coherent, if $w$ is an even function and
its monic orthogonal polynomials $p_n(w)$ satisfy the relation
\begin{equation}\label{eq:symm-coherent}
   p_n(w;x) = \frac{p_{n+1}'(w;x)}{n+1} + b_n p_{n-1}'(w;x), \qquad n \geqslant 1.
\end{equation}
More generally, we can call $w$ self-coherent if it satisfies \eqref{eq:cohen-poly}, that is,
$$
 p_n(w;x) = \frac{p_{n+1}'(w;x)}{n+1} +  a_n p_{n}'(w;x)+ b_n p_{n-1}'(w;x), \qquad n \geqslant 1
$$

If $w$ is self-coherent, we denote by $q_n(w)$ the polynomial of degree $n$ defined by
\begin{equation}\label{eq:q_n}
   q_{n} (w;x) = p_{n}(w;x) + n a_{n-1} p_{n-1}(w;x) + n b_{n-1} p_{n-2}(w;x),
    \quad n \geqslant 1,
\end{equation}
where, by convention, $p_{-1}(w;x)=0$ and we assume the last term is zero if $n=1$. It follows
directly from the definition that $q_n(w)$ is monic and
$$
     q_n'(w;x) =  n p_{n-1}(w;x).
$$
Notice that self-coherent orthogonal polynomials are essentially, up to a linear change of variable, the classical orthogonal polynomials (Jacobi, Laguerre and Hermite) as was proved in \cite{MBP}.

We now define the polynomials $Q_j^m$ of two variables by
\begin{equation}\label{eq:Q_jm}
    Q_k^n(x,y): = q_{n-k}(w_1;x) q_k(w_2;y), \qquad 0 \leqslant k \leqslant n, \quad n =0,1,\ldots.
\end{equation}
The derivatives of $Q_k^n$ can be given explicitly in terms of product orthogonal polynomials $P_j^m$
in \eqref{eq:productOP}.

\begin{lem} \label{lem:Qkn}
Let $\partial_i$ denote the $i$-th partial derivative. Then
\begin{align*}
  & \partial_1 Q_0^n(x,y) = n p_{n-1}(w_1;x) = n P_0^{n-1}(x,y) \quad \hbox{and} \quad \partial_2 Q_0^n (x,y) =0, \\
  & \partial_1 Q_n^n(x,y) =0 \quad \hbox{and} \quad \partial_2 Q_0^n (x,y) = n p_{n-1}(w_2;y) = n P_{n-1}^{n-1}(x,y).
\end{align*}
Furthermore, for $1 \leqslant k \leqslant n-1$,
\begin{align*}
  \partial_1 Q_k^n &  = (n-k) \left( P_k^{n-1} +k a_{k-1}(w_2) P_{k-1}^{n-2} + k b_{k-1}(w_2) P_{k-2}^{n-3}\right),  \\
  \partial_2 Q_k^n & = k \left(P_{k-1}^{n-1} + (n-k) a_{n-k-1}(w_1) P_{k-1}^{n-2} +(n-k) b_{n-k-1}(w_1) P_{k-1}^{n-3}\right).
\end{align*}
\end{lem}

\begin{proof}
For $1 \leqslant k \leqslant n$, it follows directly from the definition of $Q_k^n $ that
$$
 \partial_1 Q_k^n(x,y) = q_{n-k}'(w_1;x) q_k(w_2;y) = (n-k) p_{n-k-1}(w_1;x) q_k(w_2;y).
$$
Substituting $q_k(w_2;y)$  by its definition \eqref{eq:q_n}, the identity for $\partial_1 Q_k^n$ follows from
the definition of $P_j^m$.  The other identities are proved similarly.
\end{proof}

Let $\QQ_n = (Q_0^n, \ldots, Q_n^n)^\tr$ and $\SS_n = (S_0^n, \ldots, S_n^n)^\tr$ denote the column
vector of polynomials $Q_k^n$ and $S_k^n$, respectively. Furthermore, let $e_i$ denote the standard Euclidean
coordinate vector whose $i$-th element is 1 and all other elements are 0.

\begin{thm} \label{thm:Q=S}
For $0 \leqslant k \leqslant n$, there exist real numbers $a_{i,k}$ and $b_{i,k}$ such that
\begin{equation} \label{eq:Qkn=Skn}
  Q_k^n(x,y) \c= S_k^n(x,y) + \sum_{i=0}^{n-1} a_{i,k} S_i^{n-1}(x,y) +  \sum_{i=0}^n b_{i,k} S_i^{n-2}(x,y).
\end{equation}
Moreover, in the case of $k =0$ and $k=n$, we have, respectively,
\begin{equation} \label{eq:Q0n=S0n}
  S_0^n (x,y) \c= Q_0^n(x,y) \quad \hbox{and} \quad  S_n^n (x,y) \c= Q_n^n(x,y).
\end{equation}
In terms of vector notation, \eqref{eq:Qkn=Skn} can be written as
\begin{equation} \label{eq:Qn=Sn}
  \QQ_n \c= \SS_n + \Ab_{n-1} \SS_{n-1} + \Bb_{n-2} \SS_{n-2},
\end{equation}
where $\Ab_{n-1}$ and $\Bb_{n-2}$ are matrices of the form
$$
\Ab_{n-1}=\left[ \begin{array}{ccc}
0 & \dots & 0 \\ \hline  & & \\
& \wt{\Ab}_{n-1} & \\
 & & \\ \hline 0 & \dots & 0
\end{array}
\right] \quad \hbox{and} \quad \Bb_{n-2}=\left[ \begin{array}{ccc}
0 & \dots & 0 \\ \hline  & & \\
& \wt{\Bb}_{n-2} & \\
 & & \\ \hline 0 & \dots & 0
\end{array}
\right].
$$
Here $\wt \Ab_{n-1}$ and $\wt \Bb_{n-2}$ are matrices of size $(n-1)\times n$ and $(n-1) \times (n-1),$ respectively.
\end{thm}

\begin{proof}
If $k=0$ and $P$ is any polynomial in $\Pi_{n-1}^2$, then, by Lemma \ref{lem:Qkn},
$$
   \la Q_0^n, P \ra_\nabla  =  \la P_0^{n-1}, \partial_1 P \ra_W = 0.
$$
Since the space $\{\partial_1 P: P \in  \Pi_{n-1}^2\}$ is $\Pi_{n-2}^2$, this shows that
$Q_0^n \in \CV_n^2(\nabla)$ and it is equal to $S_0^n$ as it is monic. The proof for $S_n^n$ is
similar. Moreover, if $1 \leqslant k \leqslant n$, it follows from Lemma \ref{lem:Qkn} that
$$
 \la  Q_k^n, P \ra_\nabla  =  \la \partial_1 Q_k^{n}, \partial_1 P \ra_W+  \la \partial_2 Q_k^{n}, \partial_2 P \ra_W =0
$$
for any polynomial $P$ of degree at most $n-3$. Consequently,  $Q_k^n$ can be written as a linear combination of
the Sobolev orthogonal polynomials of degree $n, n-1$ and $n-2$. Since both $Q_k^n$ and $S_k^n$ are monic by
definition, \eqref{eq:Qkn=Skn} follows.
\end{proof}

To determine the matrices $\Ab_{n-1}$ and $\Bb_{n-2}$, we need to work with specific weight functions. The
simplest cases are the product Laguerre polynomials for which $\Bb_{n-2} =0$ and the product Gegenbauer
polynomials for which $\Ab_{n-1} =0$. These two cases will be worked out in detail in the next two sections.

\section{The product Laguerre weight}
\setcounter{equation}{0}

In this section we consider the product of Laguerre weight functions and the inner product
\eqref{eq:Laguerre}. The Laguerre polynomials are defined by (cf. \cite[Chapt V]{Szego})
\begin{eqnarray*}
L_n^{\a}(x) := \frac{(\a+1)_n}{n!}{}_1F_1(-n;\a+1;x) = \frac{(-1)^n}{n!} x^n + \cdots
\end{eqnarray*}
and their orthogonality is given by
$$
\la L_n^{\a}, L_m^{\a} \ra_{w_\a} := \frac{1}{\Gamma(\a+1)}\int_{0}^{+\infty} L_n^{\a}(x)L_m^\a(x)\,w_\a(x) dx
  = \frac{(\a+1)_n}{n!} \d_{n,m},
$$
where $(a)_n = a(a+1) \cdots (a+n-1),$ $n\geqslant 1,$ $(a)_0= 1,$   is the Pochhammer symbol. Furthermore, they satisfy the relation
(\cite[p. 102]{Szego})
$$
L_n^{\a}(x) = -\frac{d}{dx}\, L_{n+1}^{\a}(x) + \frac{d}{dx} \, L_n^{\a}(x),
$$
which shows that the Laguerre weight function $w_\a$ is self-coherent. Monic Laguerre orthogonal polynomial
$p_n(w_\a)$ and its $L^2$ norm are given by
$$
    p_n(w_\a; x) := (-1)^n \,n!\, L_n^{\a}(x), \qquad  h_n^\a: = \la  p_n(w_\a), p_n(w_\a)\ra_{w_\a} = n! \, (\a+1)_n.
$$
From these relations, it follows readily that the polynomial
$$
  q_n(w_\a; x):=  p_n(w_\a; x) + n p_{n-1}(w_\a; x)
$$
satisfies $q_n'(w_\a; x) = n p_{n-1}(w_\a;x)$ for $n =0,1,2, \ldots$

We are now ready to state our polynomials in two variables for the product Laguerre weight function
$W_{\a,\b}$ on $\RR_+^2$, with $\a,\b > -1$. We again denote the orthogonal polynomials by $P_k^n$,
$$
 P_k^n(x,y) := p_{n-k}(w_\a; x) p_k(w_\b; y),   \qquad 0 \leqslant k \leqslant n.
$$
It follows readily that these are mutually orthogonal polynomials and
\begin{equation}\label{eq:hkn-Laguerre}
  h_k^n: = \la P_k^n, P_k^n \ra_{W_{\a,\b}}  =   h^{\a}_{n-k}\,h^{\b}_{k}=(n-k)!\,k!\,(\a+1)_{n-k}\,(\b+1)_k.
\end{equation}
We also define the monic polynomial $Q_k^n$ by
$$
 Q_k^n(x,y) := q_{n-k}(w_\a; x) q_k(w_\b; y),   \qquad 0 \leqslant k \leqslant n.
$$
In this setting, their partial derivative for $1 \leqslant k\leqslant n$ in Lemma \ref{lem:Qkn} becomes the following:

\begin{lem} \label{lem:Q-Laguerre}
For $1\leqslant k\leqslant n-1$, the following formulas hold
\begin{align*}
\partial_1 \, Q_k^n(x,y) & = (n-k) \left[P_k^{n-1}(x,y) + k\, P_{k-1}^{n-2}(x,y)\right],  \\
\partial_2 \, Q_k^n(x,y) & = k \left[P_{k-1}^{n-1}(x,y) + (n-k) \,P_{k-1}^{n-2}(x,y)\right].
\end{align*}
\end{lem}

Recall that $\CV_n^2(\nabla, W_{\a,\b})$, $n \geqslant 1$, is the space of Sobolev orthogonal polynomials with respect to the bilinear form $\la \cdot, \cdot \ra_\nabla$ defined in \eqref{eq:ipd-nabla}. Let $S_k^n = x^{n-k} y^k + \cdots$ be
a monic orthogonal polynomial in $\CV_n^2(\nabla, W_{\a,\b})$. Then relation \eqref{eq:Qn=Sn} becomes
\begin{equation}\label{eq:Q-S-Laguerre}
     \QQ_n \c= \SS_n + \Ab_{n-1} \SS_{n-1}.
\end{equation}
Our goal is to show how $\Ab_{n-1}$ can be explicitly computed. To this end, we need explicit formulas for
the inner products of the gradients of the polynomials $Q_k^n$. In the following we write $\la \cdot,\cdot\ra = \la \cdot,\cdot\ra_{W_{\a,\b}}$.

\begin{lem}\label{lem:Qoc} For $0\leqslant i \leqslant n$ and $0\leqslant l\leqslant m$,
\begin{align*}
\la  Q^n_i,  Q^m_l \ra_\nabla
  =& \left[ l (m-l)^2 \, h^{m-2}_{l-1} \, \d_{i,l-1} + l^2 (m-l) \, h^{m-2}_{l-1}  \, \d_{i,l} \right] \d_{n,m-1} \\
  &+  \left[(m-l)^2 \, h^{m-1}_{l} \, \d_{i,l} +2 l^2 (m-l)^2 \, h^{m-2}_{l-1} \, \d_{i,l} + l^2 h^{m-1}_{l-1} \, \d_{i,l} \right] \d_{n,m} \\
  & +  \left[ (l+1) (m-l)^2 \, h^{m-1}_{l} \, \d_{i-1,l} + l^2 (m+1-l)\, h^{m-1}_{l-1} \, \d_{i,l} \right] \d_{n,m+1}.  \
\end{align*}
In particular,
\begin{align*}
\la Q^n_0,  Q^m_l \ra_\nabla &= (m-1)^2 \, h^{m-2}_{0} \, \d_{l,1} \, \d_{n,m-1} + m^2 \, h^{m-1}_{0}\, \d_{l,0} \, \d_{n,m}, \\
\la Q^n_n, Q^m_l \ra_\nabla &= (m-1)^2 \, h^{m-2}_{m-2} \, \d_{l,n} \, \d_{n,m-1} + m^2 \, h^{m-1}_{m-1} \, \d_{l,n} \, \d_{n,m}.
\end{align*}
\end{lem}

\begin{proof}
Directly from the definition,
$$
\la Q^n_i, Q^m_l \ra_\nabla = \la \nabla Q^n_i, \nabla Q^m_l \ra =\la \partial_1 Q^n_i, \partial_1 Q^m_l \ra + \la \partial_2 Q^n_i, \partial_2 Q^m_l \ra.
$$
By Lemmas \ref{lem:Qkn} and \ref{lem:Q-Laguerre}, the inner product $\la \partial_j Q^n_i, \partial_j Q^m_l \ra$
can be computed by the orthogonality of $P_k^n$ and \eqref{eq:hkn-Laguerre}. For example,
\begin{align*}
\la \partial_1 Q^n_i, \partial_1 Q^m_l \ra = &  (n-i) (m-l) \, \la  P^{n-1}_{i}, P^{m-1}_{l} \ra
      + l (n-i) (m-l) \, \la  P^{n-1}_{i}, P^{m-2}_{l-1} \ra  \\
 & +  i (n-i) (m-l) \, \la  P^{n-2}_{i-1}, P^{m-1}_{l} \ra + i l (n-i) (m-l) \, \la  P^{n-2}_{i-1}, P^{m-2}_{l-1} \ra \\
 = & (n-i) (m-l) \, h^{n-1}_{i} \, \d_{i,l} \, \d_{n,m} + l (n-i) (m-l) \, h^{n-1}_{i}  \, \d_{i,l-1} \, \d_{n,m-1} \\
 & +  i (n-i) (m-l) \, h^{n-2}_{i-1} \, \d_{i-1,l}  \, \d_{n-1,m} + i l (n-i) (m-l) \, h^{n-2}_{i-1} \, \d_{i,l}  \, \d_{n,m}
\end{align*}
The other terms are computed similarly.
\end{proof}

\begin{cor}\label{cor:Qni-Qml}
For $0\leqslant i \leqslant n$, $0\leqslant l\leqslant m$, and $m\leqslant n-1$ it holds
$$
\la  Q^n_i,  Q^m_l \ra_\nabla
  =  \left[ (l+1) (m-l)^2 \, h^{m-1}_{l} \, \d_{i-1,l} + l^2 (m+1-l)\, h^{m-1}_{l-1} \, \d_{i,l} \right] \d_{n-1,m}.\\
$$
In particular,
\begin{align*}
\la  Q^n_0,   Q^m_l \ra_\nabla = 0 \quad \hbox{and} \quad \la  Q^n_n, Q^m_l \ra_ \nabla = 0, \qquad m<n.
\end{align*}
\end{cor}

To determine the matrix $\Ab_{n-1}$, we will need explicit forms of the following two matrices:
$$
  \Cb_n: = \la \QQ_{n+1},\QQ_n^\tr \ra_{\nabla} \quad \hbox{and}\quad \Db_n: = \la \QQ_n,\QQ_n^\tr \ra_{\nabla}.
$$

\begin{lem}
For $n =0,1,2,\ldots$, $\Db_n$ is a diagonal matrix
\begin{equation}\label{Dn}
\Db_n= \mathrm{diag}\{d_0^n, d_1^n, \ldots, d_n^n\},
\end{equation}
where
\begin{align*}
 d_j^n  = (n-j)^2 \, h_j ^{n-1} + j^2 \, h_{j-1}^{n-1} + 2j^2 (n-j)^2 \, h_{j-1}^{n-2}, \quad  0 \leqslant j \leqslant n,
\end{align*}
with $h_j^m$ as given in \eqref{eq:hkn-Laguerre}, and $\Cb_n: (n+2) \times (n+1)$ is a bidiagonal matrix,
\begin{equation}\label{eq:matrixC}
\Cb_n =
 \left[ \begin{matrix}
 0 & 0 &  & \cdots & 0 \\
c_{1,0}^n & c_{1,1}^n &  &  & \\
 & c_{2,1}^n & c_{2,2}^n &  &  \\
   &  & \ddots & \ddots &    \\
   &   &  & c_{n,n-1}^n &  c_{n,n}^n   \\
 0 &  \cdots & & 0 & 0
\end{matrix}
\right],
\end{equation}
where
\begin{align*}
c_{i,i}^n &= i^2 (n-i+1)\, h_{i-1}^{n-1}, \qquad  1\leqslant i \leqslant n,\\
c_{i+1,i}^n &=  (i+1) (n-i)^2\, h_{i}^{n-1}, \qquad 0\leqslant i \leqslant n-1.
\end{align*}
\end{lem}

\begin{proof}
The formula for $\Db_n$ follows directly from Lemma \ref{lem:Qoc}.  Furthermore, by
Corollary \ref{cor:Qni-Qml}, for $1\leqslant i\leqslant n-1$,
$$
\la \nabla Q^n_i, \nabla Q^{n-1}_l \ra
  =   (l+1) (n-1-l)^2 \, h^{n-2}_{l} \, \d_{i,l+1} + l^2 (n-l)\, h^{n-2}_{l-1} \, \d_{i,l},
$$
which shows that $\Cb_n$ is a bidiagonal matrix and its first and the last row are zero.
\end{proof}

We are now ready to determine the matrix $\Ab_{n-1}$ in \eqref{eq:Q-S-Laguerre}.

\begin{thm} \label{thm_3.5}
Let $\Hb_n^\nabla: = \la \SS_n, \SS_n^{\tr} \ra_\nabla$. Then  $\Hb_n^\nabla$ satisfies the recursive relation
\begin{align}
 \Hb_n^{\nabla}&= \Db_n  - \Cb_{n-1} (\Hb_{n-1}^{\nabla})^{-1} \Cb_{n-1}^{\tr} \label{eq:H-rec-Lag},
 \end{align}
where the iteration is initiated by $\Hb_1^\nabla = \Ib$, the identity matrix.
Furthermore, for $n =1,2,\ldots$,
the matrix $\Ab_n$ in \eqref{eq:Q-S-Laguerre} is determined by
\begin{align}
 \Ab_n  = \Cb_{n} (\Hb_n^{\nabla})^{-1}. \label{eq:A-rec-Lag}
\end{align}
\end{thm}

\begin{proof}
Using the orthogonality of $\SS_n$ and
the fact that $S_k^n - Q_k^n \in \Pi_{n-1}^2$, we obtain from \eqref{eq:Q-S-Laguerre} that
\begin{align*}
 \la  \SS_{n+1},   \SS_{n}^{\tr}\ra_\nabla  &= \la \QQ_{n+1}, \QQ_{n}^{\tr}\ra_\nabla  - \Ab_n \la \SS_{n}, \SS_{n}^{\tr}\ra_\nabla \\
& = \la \QQ_{n+1},  \QQ_{n}^{\tr}\ra_\nabla  - \Ab_n \la   \QQ_{n},  \SS_{n}^{\tr}\ra_\nabla \\
& = \la  \QQ_{n+1},  \QQ_{n}^{\tr}\ra_\nabla - \Ab_n \la  \QQ_{n},  (\QQ_{n}-\Ab_{n-1}\,\SS_{n-1})^\tr \ra_\nabla,
\end{align*}
where we have used \eqref{eq:Q-S-Laguerre} once more . Hence, it follows that
\begin{align*}
 \la  \SS_{n+1},   \SS_{n}^{\tr}\ra_\nabla  &= \la \QQ_{n+1}, \QQ_{n}^{\tr}\ra_\nabla -  \Ab_n \la   \QQ_{n},  \QQ_{n}^{\tr}\ra_\nabla
   +  \Ab_n \la   \QQ_{n},   \SS_{n-1}^{\tr}\ra_\nabla \Ab_{n-1}^{\tr} \\
&= \la \QQ_{n+1}, \QQ_{n}^{\tr}\ra_\nabla - \Ab_n \la \QQ_{n},\QQ_{n}^{\tr}\ra_\nabla +  \Ab_n \la \QQ_{n}, \QQ_{n-1}^{\tr}\ra_\nabla \Ab_{n-1}^{\tr}.
\end{align*}
Consequently, from $\la \nabla \SS_{n+1}, \nabla \SS_{n}^{\tr}\ra=0$ we obtain
\begin{equation} \label{An-rec}
 \Ab_n \left[\la \QQ_{n},  \QQ_{n}^{\tr}\ra_\nabla - \la \QQ_{n}, \QQ_{n-1}^{\tr}\ra_ \nabla \Ab_{n-1}^{\tr} \right]=
    \la \QQ_{n+1},  \QQ_{n}^{\tr}\ra_\nabla.
\end{equation}
Next we compute $\Hb_n^{\nabla} = \la \SS_n, \SS_n^{\tr} \ra_\nabla$ by using \eqref{eq:Q-S-Laguerre} and the orthogonality of $\SS_n$,
\begin{align} \label{Hn-rec}
 \Hb_n^{\nabla} & = \la  \QQ_n,  \SS_n^{\tr} \ra_ \nabla = \la \QQ_n, (  \QQ_n-\Ab_{n-1} \SS_{n-1})^{\tr} \ra _\nabla \\
  & =   \la  \QQ_n,  \QQ_n^{\tr} \ra_\nabla - \la  \QQ_n,   \SS_{n-1}^{\tr} \ra_\nabla \Ab_{n-1}^{\tr}  \notag \\
  & = \la\QQ_n, \QQ_n^{\tr} \ra_ \nabla - \la \QQ_n,  \QQ_{n-1}^{\tr} \ra_\nabla \Ab_{n-1}^{\tr}. \notag
\end{align}
Since $\Hb_n^\nabla$ is nonsingular, substituting the above relation into \eqref{An-rec} proves \eqref{eq:A-rec-Lag}.
Furthermore, substituting \eqref{eq:A-rec-Lag} into \eqref{Hn-rec} shows that $\Hb_n^{\nabla}$ satisfies
\begin{align*}
\Hb_n^{\nabla}& = \la\QQ_n, \QQ_n^{\tr} \ra_ \nabla - \la \QQ_n,  \QQ_{n-1}^{\tr} \ra_\nabla
( \la\QQ_{n},\QQ_{n-1}^\tr \ra_\nabla (\Hb_{n-1}^{\nabla})^{-1})^{\tr},
\end{align*}
which simplifies to \eqref{eq:H-rec-Lag} from the symmetry of $\Hb_{n-1}^{\nabla}$, and therefore completes the proof.
\end{proof}

The theorem shows that $\Hb_n^\nabla$, hence $\Ab_n$, can be determined iteratively.

Since $S_0^n = Q_0^n$ and $S_n^n = Q_n^n$, we only need to determine $S_k^n$ for $1 \leqslant k \leqslant n-1$. This
additional information is reflected in the matrix structure, as shown in Theorem \ref{thm:Q=S} and \eqref{eq:matrixC},
$$
\Ab_{n-1}=\left[
\begin{array}{ccc} 0 & \dots & 0 \\ \hline  & & \\
& \wt{\Ab}_{n-1} & \\  & & \\ \hline 0 & \dots & 0
\end{array} \right]
\quad \hbox{and} \quad
 \Cb_{n-1} = \left[
\begin{array}{ccc} 0 & \dots & 0 \\ \hline  & & \\
& \wt \Cb_{n-1} & \\
 & & \\ \hline 0 & \dots & 0
\end{array} \right],
$$
where $\wt{\Ab}_{n-1}$ and $\wt \Cb_{n-1}$ are matrices of size $(n-1)\times n$. These suggest a further
simplification in the iteration, which we now explore.

The matrix structure shows that
$$
\Hb_{n}^{\nabla} = \Db_n - \Cb_{n-1} \Ab_{n-1}^{\tr} \\
  = \left[
\begin{array}{ccc}
d_0^n & & \\
 & \wt{\Db}_{n} &  \\
& & d_n^{n}
\end{array}
\right] -  \left[
\begin{array}{ccc}
0 & \cdots & 0 \\
\vdots & \wt \Cb_{n-1} \wt{\Ab}_{n-1}^{\tr} & \vdots \\
0 & \cdots & 0
\end{array}
\right],
$$
which shows that the matrix $\Hb_n^\nabla$ takes the form
\begin{equation} \label{eq:htHn}
  \Hb_n^\nabla = \left [ \begin{matrix} d_0^n & & 0 \\
      & \wh \Hb_n^\nabla & \\   0 & & d_n^n   \end{matrix} \right]
        \quad \hbox{with} \quad \wh{\Hb}_{n}^{\nabla} = \wt{\Db}_n - \wt \Cb_{n-1} \wt{\Ab}_{n-1}^{\tr}.
\end{equation}
Consequently, we only need to determine $\wh \Hb_n^\nabla$. Let us further write
$$
\wt \Cb_n = \left[\begin{array}{c|c|c}
  c_{1,0}^n &  & 0 \\
\vdots & \wh{\Cb}_{n} & \vdots \\
0 &  & c_{n,n}^{n}
\end{array}
\right] \quad \hbox{with}\quad \wh{\Cb}_{n} =
\left[\begin{matrix}
 c_{1,1}^n &  &  & \bigcirc \\
  c_{2,1}^n & c_{2,2}^n & &  \\
   & \ddots & \ddots  & \\
  & &  c_{n-1,n-2}^n & c_{n-1,n-1}^n \\
\bigcirc & & & c_{n,n-1}^n
\end{matrix}\right].
$$
It then follows from $\Ab_n = \Cb_{n} \left(\Hb_{n}^{\nabla}\right)^{-1}$ at \eqref{eq:A-rec-Lag} that
$$
\wt{\Ab}_{n} = \wt \Cb_{n}
\left[
\begin{matrix}
(d_0^n)^{-1} & \ldots & 0 \\
 & \left(\wh{\Hb}_{n}^{\nabla}\right)^{-1} &  \\
0 & \ldots & (d_n^{n})^{-1}
\end{matrix}
\right] = \left[ \begin{array}{c|c|c}
1 &  & 0 \\ \vdots & \wh{\Cb}_{n} \left(\wh{\Hb}_{n}^{\nabla}\right)^{-1} & \vdots \\ 0 &  & 1
\end{array}
\right],
$$
where we have used the fact that $c_{1,0}^{n} = d_0^n = n^2 h_0^{n-1}$ and $c_{n,n}^{n} = d_n^n = n^2 h_{n-1}^{n-1}$,
which follow directly from their explicit formulas. Consequently, we see that $\wt \Ab_n$ is of the form
\begin{equation} \label{eq:htAn}
\wt{\Ab}_{n} = \left[ \eb_1 | \wh{\Ab}_{n} | \eb_n \right] \quad \hbox{with} \quad
   \wh{\Ab}_{n} = \wh{\Cb}_{n} \left(\wh{\Hb}_{n}^{\nabla}\right)^{-1},
\end{equation}
where $\eb_1, \eb_{n}$ are, respectively, the first and the last vector in the canonical basis of $\RR^{n}$.
Consequently, it follows that
$$
\wt \Cb_{n-1} \wt{\Ab}_{n-1}^{\tr}  =d_{0}^{n-1} \eb_1 \eb_1^{\tr} +  \wh{\Cb}_{n-1} \wh{\Ab}_{n-1}^{\tr} +
   d_{n-1}^{n-1} \eb_{n-1} \eb_{n-1}^{\tr}.
$$
We finally conclude by  \eqref{eq:htHn} that the matrix $\wh \Hb_n^{\nabla}$ satisfies the relation
$$
   \wh{\Hb}_{n}^{\nabla} = \wh \Db_n  - \wh \Cb_{n-1} \wh{\Ab}_{n-1}^{\tr},
$$
where $\wh \Db_n$ is the diagonal matrix
$$
\wh{\Db}_{n} = \wt \Db_n - d_{0}^{n-1} \eb_1 \eb_1^{\tr} -  d_{n-1}^{n-1} \eb_{n-1} \eb_{n-1}^{\tr}.
$$
Summing up, we have proved the following proposition.

\begin{prop}
Let $\wh \QQ_n : =  (Q_1^n, \ldots, Q_{n-1}^n)$ and $\wh \SS_n : =
(S_1^n, \ldots, S_{n-1}^n)$. Then $\wh \Hb_n^\nabla = \la \wh \SS_n, \wh \SS_n^\tr \ra_\nabla$.
Furthermore, for $n =2,3,\ldots$,
\begin{equation} \label{eq:whQ=whS}
   \wh \QQ_n \c=  \wh \SS_n + \left[ \eb_1 \big \vert  \wh{\Ab}_{n-1} \big \vert  \eb_{n-1} \right] \SS_{n-1},
\end{equation}
where the matrices $\wh \Ab_n$ of size $n \times (n-1)$ and $\wh \Hb_n^{\nabla}$ of size $(n-1)\times(n-1)$ are determined
iteratively by
$$
 \wh{\Ab}_{n} = \wh{\Cb}_{n} \big(\wh{\Hb}_{n}^{\nabla}\big)^{-1} \quad \hbox{and}\quad
 \wh{\Hb}_{n}^{\nabla} = \wh \Db_n  - \wh \Cb_{n-1} \wh{\Ab}_{n-1}^{\tr}
$$
for $n=3,4,\ldots,$ with the starting point $\wh \Ab_1 = 0$.
\end{prop}

\begin{exam}
In the case of $\a = \b =0$, the iterative algorithm gives
\begin{align*}
\wh \Ab_2 & = \left[\begin{matrix} 1 \\ 1 \end{matrix}\right], \quad
\wh \Hb_2 =\left[\begin{matrix} 2 \end{matrix}\right], \\
\wh \Ab_3 & = \f 1 4\left[\begin{matrix} 5 & 1 \\ 5 & 5 \\ 1 & 5 \end{matrix}\right], \quad
\wh \Hb_3 =\left[\begin{matrix} 10 & -2 \\ -2 & 10 \end{matrix}\right], \\
\wh \Ab_4 & = \f 1{56}\left[\begin{matrix} 90 & 24 & 6 \\ 53 & 72 &11 \\ 11 & 72 & 53 \\
    6 & 24 & 90  \end{matrix}\right], \quad
\wh \Hb_4 = \left[\begin{matrix} 93 & -12 & -3 \\ -12 & 48 & -12 \\ -3 & -12 & 93 \end{matrix}\right].
\end{align*}
\end{exam}

Once the matrices $\wh \Ab_n$ are determined, the relation \eqref{eq:whQ=whS} can be used to
determine the Sobolev orthogonal polynomials $\SS_n$ iteratively, since
$$
   \wh \SS_n \c=  \wh \QQ_n - Q_0^{n-1} \eb_1 -   Q_{n-1}^{n-1} \eb_{n-1} - \wh{\Ab}_{n-1} \wh \SS_{n-1},
$$
where we have used $S_0^{n-1} = Q_0^{n-1}$ and $S_{n-1}^{n-1} = Q_{n-1}^{n-1}$.

We could also determine the polynomials $S_k^n$ directly by solving a linear system of equations. For this
purpose, we fix $k$, $1\leqslant k\leqslant n-1$, write
\begin{equation}\label{eq:Skn2=Qkn2}
     S^n_k(x,y) \c= Q^n_k(x,y)+\sum_{j=1}^{n-1} \sum_{i=0}^j a_i^j \, Q_i^j(x,y)
\end{equation}
and determine the coefficient $a_i^j$ by the orthogonality $\la S_k^n, Q_j^m \ra_\nabla =0$ for
$0 \leqslant l \leqslant m \leqslant n-1$, which is equivalent to the linear system of equations
$$
\sum_{j=1}^{n-1} \sum_{i=0}^j a_i^j \, \la  Q_i^j,  Q^m_l \ra_\nabla
    = - \la   Q_k^n,   Q^m_l \ra_\nabla, \quad  0 \leqslant l \leqslant m \leqslant n-1.
$$
By Lemma \ref{lem:Qoc}, these equations become
\begin{align*}
& l(m-l)^2 h_{l-1}^{m-2}\, a_{l-1}^{m-1} + l^2(m-l) h_{l-1}^{m-2}\, a_l^{m-1} \\
&\qquad +\left[(m-l)^2 h_l^{m-1}+2(m-l)^2 l^2 h_{l-1}^{m-2}+l^2 h_{l-1}^{m-1}\right]\, a_l^m \\
&\qquad + l^2(m-l+1)h_{l-1}^{m-1}\, a_l^{m+1}+(l+1)(m-l)^2h_l^{m-1} \, a_{l+1}^{m+1}\\
&=-[(l+1)(m-l)^2h_l^{m-1} \d_{k-1,l}+l^2(m+1-l)h_{l-1}^{m-1}\d_{k,l}]\d_{m,n-1}.
\end{align*}
Observe that for $m=n-1$ the third term in the left hand side does not appear since $a_{l}^{n} = 0$ by definition.
Using $h_{l-1}^{m-1}=(m-l)(\a +m-l) h_{l-1}^{m-2}$ and $h_l^{m-1}=l(\b +l) h_{l-1}^{m-2}$, the above
equations can be simplified to
\begin{align}\label{eq:rr}
 &(m-l)\, a_{l-1}^{m-1} + l\, a_l^{m-1} +\left[l \a + (m-l)\b +4l(m-l) \right]\, a_l^m \\
& \qquad + l(m-l+1)(\a +m-l) \, a_l^{m+1}+(l+1)(m-l) (\b +l) \, a_{l+1}^{m+1} \notag \\
 &= -[(l+1)(m-l) (\b + l) \d_{k,l+1}+l(m+1-l)(\a +m-l) \d_{k,l}]\d_{m,n-1}. \notag
\end{align}
The indexes of $a_l^m$ are lattices in $\Lambda_n: = \{(l,m): 0 \leqslant l \leqslant m \leqslant n-1\}$. For each $(l,m)$, the
equation \eqref{eq:rr} involves $a_l^m$ and its four neighbors, directly above and below, left
and right of $a_l^m$ in the lattice. In particular, for $l=0$ and $l=m$ we obtain the equations
\begin{align*}
 a_0^m +  a_1^{m+1}  =  - \d_{k,1} \,\d_{m,n-1},\qquad
 a_m^m +  a_m^{m+1}  =  - \d_{k,m} \, \d_{m,n-1}.
\end{align*}
By $a_{l}^{n}=0$, these equations can be written in an equivalent way as
\begin{align} \label{eq:ic}
\begin{split}
 & a_0^{n-1}  = -\d_{k,1} \qquad  a_{n-1}^{n-1} = -\d_{k,n-1} \quad  a_m^m +  a_m^{m+1}  =  0\\
 & a_0^m +  a_1^{m+1} = 0  , \qquad 1\leqslant m \leqslant n-2.
 \end{split}
\end{align}
These provide the boundary relations for the lattice $\Lambda_n$. Together, \eqref{eq:rr} and \eqref{eq:ic}
form a linear system of equations that can be solved for $\{a_l^m: 0\leqslant l\leqslant m \leqslant n-1\}$.  Furthermore,
the relations in \eqref{eq:ic} allow us to combine some of the terms in the sum \eqref{eq:Skn2=Qkn2}.
We summarize the above consideration into the following proposition.

\begin{prop}
For $1\leqslant k \leqslant n-1$, the monic Sobolev polynomials are given by
\begin{align*}
S^n_k  \c= & Q_k^n-\d_{k,1}\, Q_0^{n-1}-\d_{n,n-1}\, Q_{n-1}^{n-1}+ \sum_{j=1}^{n-2}a_0^j\,(Q_0^j-Q_1^{j+1}) \\
 & +   \sum_{j=1}^{n-2}a_j^j\,(Q_j^j-Q_j^{j+1})+  \sum_{j=4}^{n-1}\sum_{i=4}^j a_{i-2}^j \,Q_{i-2}^j
\end{align*}
where the coefficients $a_i^j$ are solutions of \eqref{eq:rr} and \eqref{eq:ic}.
\end{prop}

\begin{exam}
For the case of $\a=\b=0$, the monic Laguerre--Sobolev orthogonal polynomials satisfy the relation
$$
   S_{n-k}^n (x,y) = S_k^n(y,x), \qquad 0 \leqslant k \leqslant n.
$$
The following are these polynomials in lower degrees: $S_0^1(x,y) = x$,
\begin{align*}
  &S_0^2(x,y) = x(x-2), \quad  S_1^2(x,y) = x y -x-y,  \\
  &S_0^3(x,y) = x(x^2-6x +6), \quad S_1^3(x,y) = x^2 y - x^2 - 3 x y +3 x +y.
\end{align*}
\end{exam}

\begin{rem}
In the case of $\a=\b=0$ we have (see equation (5.2.1) in \cite{Szego})
$$
q_n(u_0;x) = (-1)^n \frac{1}{n!} L_n^{-1}(x) = (-1)^{n-1} \frac{1}{(n-1)!} x L_{n-1}^{1}(x),
$$
and therefore the constant term in $q_n(u_0;x)$ always vanishes for $n \geqslant 1$. Consequently,
in this case, equations that hold under \textit{modulo constant}, or $\c=$, in Theorem \ref{thm:Q=S} can
be replaced by the usual equal sign. 
\end{rem}

\section{The product Gegenbauer weight}
\setcounter{equation}{0}

In this section we study the product of Gegenbauer (or ultraspherical) weight functions
and the inner product \eqref{eq:Gegenbauer}. Let
$$
u_{\a}(x) := (1 - x^2)^{\a - \frac{1}{2}}, \qquad \a > -\tfrac{1}{2}.
$$
The classical Gegenbauer polynomials $C_n^\l$, defined by (\cite[Chapt IV]{Szego})
$$
C_n^{\a}(x) := \binom{n+2\a-1}{n}{}_2F_1(-n,n+2\a;\a+\frac{1}{2};x) = 2^n\binom{n+\a-1}{n} \,x^n + \cdots
$$
are orthogonal with respect to the inner product
$$
  \la f, g \ra_{u_\a} := \frac{\Gamma(\a+1)}{\Gamma(\a+1/2)\Gamma(1/2)}\int_{-1}^{1} f(x)g(x)\,u_\a(x) dx.
$$
More precisely, they satisfy
\begin{align*}
\la C_n^{\a}, C_m^{\a} \ra_{u_\a} = \frac{2^{1-2\a}\,\a\,\sqrt{\pi}}{\Gamma(\a+1/2)\,\Gamma(\a)}\,\frac{\Gamma(n+2\a)}{(n+\a)\,n!} \d_{n,m}.
\end{align*}
Moreover, they are self-coherent since they satisfy (\cite[(4.7.29) in p. 83]{Szego})
$$
2\,(n+\a) \,C_n^{\a}(x) = \frac{d}{dx}\, \left[C_{n+1}^{\a}(x) - C_{n-1}^{\a}(x)\right],\quad n\ge1.
$$
Monic Gegenbauer orthogonal polynomials $p_n(u_\a)$ are defined by
$$
p_n(u_\a; x) := 2^{-n} \,\binom{n+\a-1}{n}^{-1}\, C_n^{\a}(x),$$
and their $L^2$ norms are given by
$$ h_n^\a: = \la  p_n(u_\a), p_n(u_\a)\ra_{u_\a} = \frac{2^{1-2\a-2n}\,\sqrt{\pi}\,n!\,\Gamma(\a+1)\,
\Gamma(n+2\a)}{\Gamma(\a+1/2)\,\Gamma(n+\a)\,\Gamma(n+\a+1)}.
$$
From these relations, we deduce that the polynomial
$$
  q_n(u_\a; x):=  p_n(u_\a; x) + n\,b_{n-1}(\alpha)\, p_{n-2}(u_\a; x),
$$
where
$$
b_{n-1}(\alpha) = - \frac{(n-1)}{4\,(n+\a-1)\,(n+\a-2)}, \quad n\ge 2,
$$
satisfies $q_n'(u_\a; x) = n \, p_{n-1}(u_\a;x)$ for $n =1,2, \ldots$

We define the product Gegenbauer weight function $U_{\a,\b}(x,y):= u_\a(x) y_\b(y)$ on $[-1,1]\times [-1,1]$
for $\a,\b > -1/2$ and define monic product polynomials
$$
 P_k^n(x,y) := p_{n-k}(u_\a; x) \, p_k(u_\b; y),   \qquad 0 \leqslant k \leqslant n.
$$
These are mutually orthogonal polynomials, and
\begin{equation}\label{eq:hkn-Gegenbauer}
  h_k^n: = \la P_k^n, P_k^n \ra_{U_{\a,\b}}  =   h^{\a}_{n-k}\,h^{\b}_{k}.
\end{equation}
We also define the monic polynomial $Q_k^n$ by
$$
 Q_k^n(x,y) := q_{n-k}(u_\a; x) \,q_k(u_\b; y),   \qquad 0 \leqslant k \leqslant n.
$$
In this setting, their partial derivatives for $1 \leqslant k \leqslant n$ in Lemma \ref{lem:Qkn} become

\begin{lem} \label{lem:Q-Gegenbauer}
For $1\leqslant k\leqslant n-1$, 
\begin{align*}
\partial_1 \, Q_k^n(x,y) & = (n-k) \left[P_k^{n-1}(x,y) + k\,b_{k-1}(\beta)\, P_{k-2}^{n-3}(x,y)\right],  \\
\partial_2 \, Q_k^n(x,y) & = k \left[P_{k-1}^{n-1}(x,y) + (n-k)\,b_{n-k-1}(\alpha) \,P_{k-1}^{n-3}(x,y)\right].
\end{align*}
\end{lem}

Denote by $\CV_n^2(\nabla, U_{\a,\b})$, $n \ge 1$, the space of Sobolev orthogonal polynomials
with respect to the bilinear form $\la \cdot, \cdot \ra_\nabla$ defined in \eqref{eq:ipd-nabla}, and
let $S_k^n = x^{n-k} y^k + \cdots$ be
the monic orthogonal polynomials in $\CV_n^2(\nabla, U_{\a,\b})$. In this case, relation \eqref{eq:Qn=Sn}
becomes
\begin{equation}\label{eq:Q-S-Gegenbauer}
     \QQ_n \c= \SS_n + \Bb_{n-2} \SS_{n-2}.
\end{equation}
To compute $\Bb_{n-2}$ explicitly, we need explicit formulas for
the inner products of the gradients of the polynomials $Q_k^n$. In order to simplify the expressions,
from now on we will write $\la \cdot,\cdot\ra = \la \cdot,\cdot\ra_{U_{\a,\b}}$.


\begin{lem}\label{lem:Qoc-G} For  $0\leqslant i\leqslant n$ and $0\leqslant l \leqslant m$,
\begin{align*}
\la  Q^n_i, Q^m_l \ra_\nabla
  =& \d_{n,m+2}  \left[ (m-l)^2(l+2)\, b_{l+1}(\beta) \, h^{m-1}_{l} \, \d_{i,l+2} \right. \\
  &  \qquad \left. + l^2 (m-l+2)\, b_{m-l+1}(\alpha) \, h^{m-1}_{l-1}  \, \d_{i,l} \right] \\
  & + \d_{n,m} \left[(m-l)^2 \, h^{m-1}_{l} \, \d_{i,l} + l^2 (m-l)^2\,b^2_{l-1}(\beta) \, h^{m-3}_{l-2} \, \d_{i,l} \right.  \\
  & \qquad \left. + l^2 h^{m-1}_{l-1} \, \d_{i,l} + l^2 (m-l)^2\,b^2_{m-l-1}(\alpha) \, h^{m-3}_{l-1} \, \d_{i,l} \right]  \\
  & + \d_{n,m-2} \left[ l (m-l)^2\, b_{l-1}(\beta) \, h^{m-3}_{l-2} \, \d_{i,l-2} \right.  \\
  & \qquad \left. + l^2 (m-l)\, b_{m-l-1}(\alpha)\, h^{m-3}_{l-1} \, \d_{i,l} \right].  \
\end{align*}
In particular,
\begin{align*}
\la Q^n_0,  Q^m_l \ra_\nabla &= 2 (m-2)^2\, b_1(\beta) \, h^{m-1}_{0} \, \d_{l,2} \, \d_{n,m-2} + m^2 \, h^{m-1}_{0} \, \d_{l,0}\, \d_{n,m}.  \\
\la Q^n_n, Q^m_l \ra_\nabla &= 2(m-2)^2\,  b_1(\alpha) \, h_{m-1}^{m-1}\,\d_{l,n}\, \d_{n,m-2} + m^2 \, h^{m-1}_{m-1} \, \d_{l,n}\, \d_{n,m}.
\end{align*}
\end{lem}
The proof is analogous to that of Lemma \ref{lem:Qoc}.

\begin{cor}\label{cor:Qni-Qml-G}
For $0\leqslant i \leqslant n$, $0\leqslant l\leqslant m$, and $m\leqslant n-1$ it holds
\begin{align*}
\la  Q^n_i, Q^m_l \ra_\nabla
  =\, & \d_{n,m+2}  \left[ (m-l)^2(l+2)\, b_{l+1}(\beta) \, h^{m-1}_{l} \, \d_{i,l+2} \right. \\
  &  \qquad \left.+ l^2 (m-l+2)\, b_{m-l+1}(\alpha) \, h^{m-1}_{l-1}  \, \d_{i,l} \right].
\end{align*}
In particular, \begin{align*}
\la Q^n_0,  Q^m_l \ra_\nabla &= 0 \quad \hbox{and} \quad \la  Q^n_n, Q^m_l \ra_ \nabla = 0, \qquad m<n.
\end{align*}
\end{cor}

To determine the matrix $\Bb_{n-2}$, we will need explicit forms of the following two matrices:
$$
  \Cb_n: = \la \QQ_{n+2},\QQ_n^\tr \ra_{\nabla} \quad \hbox{and}\quad \Db_n: = \la \QQ_n,\QQ_n^\tr \ra_{\nabla}.
$$

\begin{lem}
For $n =0,1,2,\ldots$, $\Db_n$ is a diagonal matrix
\begin{equation}
\Db_n= \mathrm{diag}\{d_0^n, d_1^n, \ldots, d_n^n\},
\end{equation}
where, for $0 \leqslant j \leqslant n$,
\begin{align*}
 d_j^n  &= (n-j)^2 \, h_j ^{n-1} + j^2 (n-j)^2\, b^2_{j-1}(\beta) \, h_{j-2}^{n-3} \\
        &\quad + j^2 \, h_{j-1}^{n-1} + j^2 (n-j)^2\, b^2_{n-j-1}(\alpha) \, h_{j-1}^{n-3},
\end{align*}
with $h_j^m$ as given in \eqref{eq:hkn-Gegenbauer}, and $\Cb_n: (n+3) \times (n+1)$ is a bidiagonal matrix,
\begin{equation}\label{eq:matrixC-G}
\Cb_n =
 \left[ \begin{matrix}
 0 & 0 & 0 & \cdots & 0 \\
0 & c_{1,1}^n & 0 &  & \\
c_{2,0}^n & 0 & c_{2,2}^n &  &  \\
   & \ddots & \ddots & \ddots &    \\
     &   &  & 0 &  c_{n,n}^n   \\
   &   &  & c_{n+1,n-1}^n &  0   \\
 0 &  \cdots & & 0 & 0
\end{matrix}
\right],
\end{equation}
where
\begin{align*}
c_{l,l}^n &= \la Q^{n+2}_l, Q^n_l \ra_{\nabla} =
 l^2 (n-l+2)\, b_{n-l+1}(\a)\, h_{l-1}^{n-1}, \qquad  0 \leqslant l \leqslant n\\
c_{l+2,l}^n &= \la Q^{n+2}_{l+2}, Q^n_l \ra_{\nabla} =
 (l+2) (n-l)^2\, b_{l+1}(\b)\, h_l^{n-1}, \qquad 0\leqslant l \leqslant n.
\end{align*}
\end{lem}

\begin{proof}
The formula for $\Db_n$ follows directly from Lemma \ref{lem:Qoc-G}.  Furthermore, by
Corollary \ref{cor:Qni-Qml-G}, for $0\leqslant i\leqslant n+2$,
$$
\la Q^{n+2}_i, Q^n_l \ra_{\nabla}
  =   (l+2)(n-l)^2\, b_{l+1}(\b) \, h^{n-1}_{l} \, \d_{i,l+2} + l^2 (n-l+2)\, b_{n-l+1}(\a)\, h^{n-1}_{l-1} \, \d_{i,l},
$$
which shows that $\Cb_n$ is a bidiagonal matrix and its first and last row are zero.
\end{proof}


Now we can compute the matrix $\Bb_{n-2}$ in \eqref{eq:Q-S-Gegenbauer}.

\begin{thm}
Let $\Hb_n^\nabla: = \la \SS_n, \SS_n \ra_\nabla$. Then  $\Hb_n^\nabla$ satisfies the recursive relation
\begin{align}
 \Hb_n^{\nabla}&= \Db_n  - \Cb_{n-2} (\Hb_{n-2}^{\nabla})^{-1} \Cb_{n-2}^{\tr} \label{eq:H-rec-Geg},
 \end{align}
where the iteration is initiated by $\Hb_1^\nabla = \Ib$, the identity matrix,
and $\Hb_2^\nabla = \Db_2$.
Furthermore, for $n =1,2,\ldots$,
the matrix $\Bb_n$ in \eqref{eq:Q-S-Gegenbauer} is determined by
\begin{align}
 \Bb_n  = \Cb_{n} (\Hb_n^{\nabla})^{-1}. \label{eq:B-rec-Geg}
\end{align}
\end{thm}

\begin{proof} This is similar to the proof of Theorem~\ref{thm_3.5}. Using \eqref{eq:Q-S-Gegenbauer} twice we obtain
\begin{align*}
 \la  \SS_{n+2},   \SS_{n}^{\tr}\ra_\nabla  &= \la \QQ_{n+2}, \QQ_{n}^{\tr}\ra_\nabla  - \Bb_n \la \SS_{n}, \SS_{n}^{\tr}\ra_\nabla \\
& = \la  \QQ_{n+2},  \QQ_{n}^{\tr}\ra_\nabla - \Bb_n \la  \QQ_{n},  (\QQ_{n}-\Bb_{n-2}\,\SS_{n-2})^\tr \ra_\nabla\\
&= \la \QQ_{n+2}, \QQ_{n}^{\tr}\ra_\nabla - \Bb_n \la \QQ_{n},\QQ_{n}^{\tr}\ra_\nabla +  \Bb_n \la \QQ_{n}, \QQ_{n-2}^{\tr}\ra_\nabla \Bb_{n-2}^{\tr}.
\end{align*}
And from $\la \nabla \SS_{n+2}, \nabla \SS_{n}^{\tr}\ra=0$ we deduce
\begin{equation} \label{Bn-rec}
\la \QQ_{n+2},  \QQ_{n}^{\tr}\ra_\nabla = \Bb_n \left[\la \QQ_{n},  \QQ_{n}^{\tr}\ra_\nabla - \la \QQ_{n}, \QQ_{n-2}^{\tr}\ra_ \nabla \Bb_{n-2}^{\tr} \right].
\end{equation}
Next we compute $\Hb_n^{\nabla} = \la \SS_n, \SS_n^{\tr} \ra_\nabla$ by using \eqref{eq:Q-S-Gegenbauer} and the orthogonality of $\SS_n$,
\begin{align} \label{Hn-Geg-rec}
 \Hb_n^{\nabla} & = \la  \QQ_n,  \SS_n^{\tr} \ra_ \nabla = \la \QQ_n, (  \QQ_n-\Bb_{n-2} \SS_{n-2})^{\tr} \ra _\nabla \\
  & = \la\QQ_n, \QQ_n^{\tr} \ra_ \nabla - \la \QQ_n,  \QQ_{n-2}^{\tr} \ra_\nabla \Bb_{n-2}^{\tr}. \notag
\end{align}
Since $\Hb_n^\nabla$ is nonsingular, substituting the above relation into \eqref{Bn-rec} proves \eqref{eq:B-rec-Geg}.
Finally, substituting \eqref{eq:B-rec-Geg} into \eqref{Hn-Geg-rec} shows \eqref{eq:H-rec-Geg}.
\end{proof}

The previous theorem shows that $\Hb_n^\nabla$ and $\Bb_n$ can be determined iteratively.

Since $S_0^n \c= Q_0^n$ and $S_n^n \c= Q_n^n$, we only need to determine $S_k^n$ for $1 \leqslant k \leqslant n-1$. The matrix structure reflects this information, as shown in Theorem \ref{thm:Q=S} and \eqref{eq:matrixC-G}; in fact we have
$$
\Bb_{n-2}=\left[
\begin{array}{ccc} 0 & \dots & 0 \\ \hline  & & \\
& \wt{\Bb}_{n-2} & \\  & & \\ \hline 0 & \dots & 0
\end{array} \right]
\quad \hbox{and} \quad
 \Cb_{n-2} = \left[
\begin{array}{ccc} 0 & \dots & 0 \\ \hline  & & \\
& \wt \Cb_{n-2} & \\
 & & \\ \hline 0 & \dots & 0
\end{array} \right],
$$
where $\wt{\Bb}_{n-2}$ and $\wt \Cb_{n-2}$ are matrices of size $(n-1)\times (n-1)$.

We now proceed as in Section 3 to simplify the iteration process.

The matrix structure reads as
$$
\Hb_{n}^{\nabla} = \Db_n - \Cb_{n-2} \Bb_{n-2}^{\tr} \\
  = \left[
\begin{array}{ccc}
d_0^n & & \\
 & \wt{\Db}_{n} &  \\
& & d_n^{n}
\end{array}
\right] -  \left[
\begin{array}{ccc}
0 & \cdots & 0 \\
\vdots & \wt \Cb_{n-2} \wt{\Bb}_{n-2}^{\tr} & \vdots \\
0 & \cdots & 0
\end{array}
\right],
$$
which shows that the matrix $\Hb_n^\nabla$ takes the form
\begin{equation} \label{eq:htHn-G}
  \Hb_n^\nabla = \left [ \begin{matrix} d_0^n & & 0 \\
      & \wh \Hb_n^\nabla & \\   0 & & d_n^n   \end{matrix} \right]
        \quad \hbox{with} \quad \wh{\Hb}_{n}^{\nabla} = \wt{\Db}_n - \wt \Cb_{n-2} \wt{\Bb}_{n-2}^{\tr},
\end{equation}
and we only need to determine $\wh \Hb_n^\nabla$. If we write
$$
\wt \Cb_n = \left[\begin{array}{c|c|c}
  0  &  & 0 \\
  c_{2,0}^n &  & 0 \\
\vdots & \wh{\Cb}_{n} & \vdots \\
0 &  & c_{n,n}^{n}\\
0 &  & 0
\end{array}
\right] \quad \hbox{with}\quad
\wh{\Cb}_{n} =
 \left[ \begin{matrix}
c_{1,1}^n & 0 &  & \\
 0 & c_{2,2}^n &  &  \\
c_{3,1}^n & 0 &  & \\
& \ddots & \ddots & \ddots    \\
        &  &  0 & c_{n-1,n-1}^n \\
        &  &  c_{n,n-2}^n & 0   \\
     &   &  0 & c_{n+1,n-1}^n
\end{matrix}
\right],
$$
then from $\Bb_n = \Cb_{n} \left(\Hb_{n}^{\nabla}\right)^{-1}$ at \eqref{eq:B-rec-Geg} we conclude
$$
\wt{\Bb}_{n} = \wt \Cb_{n}
\left[
\begin{matrix}
(d_0^n)^{-1} & \ldots & 0 \\
 & \left(\wh{\Hb}_{n}^{\nabla}\right)^{-1} &  \\
0 & \ldots & (d_n^{n})^{-1}
\end{matrix}
\right] = \left[ \begin{array}{c|c|c}
0 &  & 0 \\
2 b_1(\b) &  & 0 \\ \vdots & \wh{\Cb}_{n} \left(\wh{\Hb}_{n}^{\nabla}\right)^{-1} & \vdots \\
0 &  & 2 b_1(\a) \\
0 &  & 0
\end{array}
\right],
$$
where we use
\begin{align*}
c_{2,0}^{n} &= 2 b_1(\b)\, n^2 h_0^{n-1}, \qquad d_0^n = n^2 h_0^{n-1}, \\
c_{n,n}^{n} &= 2 b_1(\a)\, n^2 h_{n-1}^{n-1}, \qquad d_n^n = n^2 h_{n-1}^{n-1}.
\end{align*}

Consequently, we see that $\wt \Bb_n$ is of the form
\begin{equation} \label{eq:htBn}
\wt{\Bb}_{n} = \left[ 2 b_1(\b) \eb_2 | \wh{\Bb}_{n} | 2 b_1(\a) \eb_{n} \right] \quad \hbox{with} \quad
   \wh{\Bb}_{n} = \wh{\Cb}_{n} \left(\wh{\Hb}_{n}^{\nabla}\right)^{-1},
\end{equation}
where $\eb_2$ and $\eb_{n}$ are, respectively, the second vector and the second last vector in the canonical basis
of $\RR^{n+1}$. Consequently, it follows that
$$
\wt \Cb_{n-2} \wt{\Bb}_{n-2}^{\tr}  = 4 b_1^2(\b)d_{0}^{n-2} \eb_2 \eb_2^{\tr} +  \wh{\Cb}_{n-2} \wh{\Bb}_{n-2}^{\tr} +
4 b_1^2(\a) d_{n-2}^{n-2} \eb_{n-2} \eb_{n-2}^{\tr}.
$$
We finally conclude by  \eqref{eq:htHn-G} that the matrix $\wh \Hb_n^{\nabla}$ satisfies the relation
$$
   \wh{\Hb}_{n}^{\nabla} = \wh \Db_n  - \wh \Cb_{n-2} \wh{\Bb}_{n-2}^{\tr},
$$
where $\wh \Db_n$ is the diagonal matrix
$$
\wh{\Db}_{n} = \wt \Db_n - 4 b_1^2(\b) d_{0}^{n-2} \eb_2 \eb_2^{\tr} -  4 b_1^2(\a) d_{n-2}^{n-2} \eb_{n-2} \eb_{n-2}^{\tr}.
$$
Summing up, we have proved the following proposition.

\begin{prop}
Let $\wh \QQ_n : =  (Q_1^n, \ldots, Q_{n-1}^n)$ and $\wh \SS_n : =
(S_1^n, \ldots, S_{n-1}^n)$. Then $\wh \Hb_n^\nabla = \la \wh \SS_n, \wh \SS_n^\tr \ra_\nabla$.
Furthermore, for $n =3,4,\ldots$,
\begin{equation} \label{eq:whQ=whS-G}
   \wh \QQ_n \c=  \wh \SS_n + \left[ 2 b_1(\b) \eb_2 \big \vert  \wh{\Bb}_{n-2} \big \vert 2 b_1(\a) \eb_{n-2} \right] \SS_{n-2},
\end{equation}
where the matrices $\wh \Bb_n$ of size $n \times (n-2)$ and $\wh \Hb_n^{\nabla}$ of size $(n-1)\times(n-1)$ are determined
iteratively by
$$
 \wh{\Bb}_{n} = \wh{\Cb}_{n} \big(\wh{\Hb}_{n}^{\nabla}\big)^{-1} \quad \hbox{and}\quad
 \wh{\Hb}_{n}^{\nabla} = \wh \Db_n  - \wh \Cb_{n-2} \wh{\Bb}_{n-2}^{\tr}
$$
for $n=3,4,\ldots,$ with the initial condition $\wh{\Bb}_{1} = 0$.
\end{prop}

\begin{exam}
In the case of $\a = \b =1$ we have $b_1(1) = -\frac{1}{8}$, and the iterative algorithm gives

\begin{align*}
\wh \Bb_2 &= -\f 1 8\left[\begin{matrix}  1 \\ 0 \\ 1 \end{matrix}\right], \quad  \wh \Hb_2 =\left[\begin{matrix} \frac{1}{2} \end{matrix}\right], \\
\wh \Bb_3 &= -\frac{1}{20} \left[\begin{matrix} 1 & 0 \\ 0 & 4 \\ 4 & 0 \\
    0 & 1  \end{matrix}\right], \quad
\wh \Hb_3 = \frac{5}{16}\left[\begin{matrix} 1 & 0 \\ 0 & 1 \end{matrix}\right] \\
\wh \Bb_4 &= -\frac{1}{880}\left[\begin{matrix} 21 & 0 & 1 \\ 0 & 110 & 0 \\ 198 & 0 & 198 \cr
    0 & 110 & 0 \\ 1 & 0 & 21 \end{matrix}\right], \quad 
\wh \Hb_4 = \frac{1}{128}\left[\begin{matrix} 21 & 0 & -1 \\ 0 & 16 & 0 \\ -1 & 0 & 21 \end{matrix}\right].
\end{align*}
\end{exam}

Once the matrices $\wh \Bb_n$ are determined, the relation \eqref{eq:whQ=whS-G} can be used to
determine the Sobolev orthogonal polynomials $\SS_n$ iteratively, since
$$
   \wh \SS_n \c=  \wh \QQ_n - Q_0^{n-2} 2 b_1(\b) \eb_2 -   Q_{n-2}^{n-2} 2 b_1(\a) \eb_{n-2} - \wh{\Bb}_{n-2} \wh \SS_{n-2},
$$
where we have used $S_0^{n} \c= Q_0^{n}$ and $S_{n}^{n} \c= Q_{n}^{n}$.


\begin{exam}
For the case of $\a=\b=1$, the monic Gegenbauer--Sobolev orthogonal polynomials satisfy the relation
$$
   S_{n-k}^n (x,y) = S_k^n(y,x), \qquad 0 \leqslant k \leqslant n.
$$
The following are these polynomials in lower degrees: 
\begin{align*}
  &S_0^1(x,y) = x\\
  &S_0^2(x,y) = x^2, \quad  S_1^2(x,y) = x y,  \\
  &S_0^3(x,y) = x (x^2-\frac{3}{4}), \quad S_1^3(x,y) = (x^2 - \frac{1}{4}) y, \\
  &S_0^4(x,y) = x^2(x^2-1), \quad S_1^4(x,y) = x (x^2  - \frac{5}{8}) y, \quad S_2^4(x,y) = x^2 y^2 - 
  \frac{1}{4} x^2 - \frac{1}{4} y^2.   
\end{align*}
\end{exam}

\begin{rem} In contrast to the Laguerre case with $\a = \b = 0$, we need the \textit{modulo constant}, or
$\c=$, in the Theorem \ref{thm:Q=S} for the Gegenbauer case. Note, however, that this is not a real limitation,
since our main goal is to construct a basis for $\CV_n^2(S)$, for which the additive constant does not 
matter, as shown in Theorem \ref{sobolev-basis}. 
\end{rem}


\begin{thebibliography}{99}

\bibitem{DX01}
       C. F. Dunkl, Y. Xu,
       \textit{Orthogonal polynomials of several variables},
       Encyclopedia of Mathematics and its Applications \textbf{81}, Cambridge
       University Press, 2001, 2nd edition, 2014.
       
\bibitem{KLJ97}
        K. H. Kwon, J. K. Lee, I. H. Jung,
        Sobolev Orthogonal Polynomials relative to   $\lambda p(c) q(c)  +  < \tau , p' q' >$,
        \textit{Comm. Korean Math. Soc.} \textbf{12} (1997), 603--617.

\bibitem{LX}
        H. Li, Y. Xu,
        Spectral approximation on the unit ball,
        arXiv:1310.2283.

\bibitem{MBP}
         F. Marcell\'an,  A. Branquinho, J. C. Petronilho,
         Classical orthogonal polynomials: A functional approach
         \textit{Acta Applicandae  Mathematicae} \textbf{34} (1994) 283--303.

\bibitem{MPP1}
        F. Marcell\'an, T. E. P\'erez, M. A. Pi\~nar,
        Gegenbauer-Sobolev orthogonal polynomials. In \textit{Proceedings Conference on NonLinear Numerical Methods and Rational Approximation II}. A. Cuyt ed., Kluwer Academic Publishers. Dordrecht. 1994. 71--82.

\bibitem{MPP2}
       F. Marcell\'an, T. E. P\'erez, M. A. Pi\~nar,
       Laguerre--Sobolev orthogonal polynomials,
       \textit{J. Comput. Appl. Math.}, \textbf{71} (2) (1996), 245--265.

\bibitem{MX}
       F. Marcell\'an, Y. Xu,
       On Sobolev orthogonal polynomials, arXiv:1403.6249

\bibitem{Szego}
       G. Szeg\H{o},
       \textit{Orthogonal polynomials},
       Amer. Math. Soc. Colloq. Publ. Vol. \textbf{23}, Amer. Math. Soc. Providence, RI, 1975. Fourth Edition.

\bibitem{X08}
        Y. Xu,
        Sobolev orthogonal polynomials defined via gradient on the unit ball,
       \textit{J. Approx. Theory}  \textbf{152} (2008), 52--65.

\end{thebibliography}
\end{document}